\documentclass[10pt,twoside]{article}
\usepackage{mathrsfs}
\usepackage{amssymb}
\usepackage{amsmath}
\usepackage[all]{xy}
\usepackage{amsthm}

\usepackage{url}
\numberwithin{equation}{section}

\newtheorem{theorem}{Theorem}[section]
\newtheorem{proposition}[theorem]{Proposition}
\newtheorem{definition}[theorem]{Definition}

\newtheorem{lemma}[theorem]{Lemma}

\newtheorem{corollary}[theorem]{Corollary}

\newtheorem{theorem*}{Theorem}

\newcommand{\Tor}{\operatorname{Tor}}
\newcommand{\Mod}{\operatorname{Mod}}
\newcommand{\Hom}{\operatorname{Hom}}
\newcommand{\Ext}{\operatorname{Ext}}
\newcommand{\A}{\mathscr{A}}

\newcommand{\ra}{\rightarrow}

\newcommand{\C}{\mathscr{C}}

\newcommand{\res}{\operatorname{res}}
\newcommand{\cores}{\operatorname{cores}}

\def\Im{\mathop{\rm Im}\nolimits}

\def\Mod{\mathop{\rm Mod}\nolimits}

\def\Prod{\mathop{\rm Prod}\nolimits}

\title{ \bf Duality Pairs Induced by One-Sided Gorenstein Subcategories 
}
\vspace{0.2cm}

\author{Weiling Song$^{a}$,
Tiwei Zhao$^{b}$, Zhaoyong Huang$^{c,}$\thanks{{\it E-mail address}: songwl@njfu.edu.cn (W. L. Song), tiweizhao@hotmail.com (T. W. Zhao),
huangzy@nju.edu.cn (Z. Y. Huang)}\\
{\footnotesize \it $^a$Department of Applied Mathematics, College of Science, Nanjing Forestry University,}\\
{\footnotesize \it Nanjing 210037, Jiangsu Province, P.R. China;}\\
{\footnotesize \it $^b$School of Mathematical Sciences, Qufu Normal University, Qufu 273165, Shandong Province, P.R. China;}\\
{\footnotesize \it $^c$Department of Mathematics, Nanjing University, Nanjing 210093, Jiangsu Province, P.R. China}}
\date{ }
\begin{document}

\baselineskip=16pt
\maketitle

\begin{abstract}
For a ring $R$ and an additive subcategory $\C$ of the category $\Mod R$ of left $R$-modules,
under some conditions we prove that the right Gorenstein subcategory of $\Mod R$ and the left
Gorenstein subcategory of $\Mod R^{op}$ relative to $\C$ form a coproduct-closed duality pair.
Let $R,S$ be rings and $C$ a semidualizing ($R,S$)-bimodule. As applications of the above result,
we get that if $S$ is right coherent and $C$ is faithfully semidualizing,
then $(\mathcal{GF}_C(R),\mathcal{GI}_C(R^{op}))$
is a coproduct-closed duality pair and $\mathcal{GF}_C(R)$ is covering in $\Mod R$, where
$\mathcal{G}\mathcal{F}_C(R)$ is the subcategory of $\Mod R$
consisting of $C$-Gorenstein flat modules and $\mathcal{G}\mathcal{I}_C(R^{op})$
is the subcategory of $\Mod R^{op}$ consisting of $C$-Gorenstein injective modules;
we also get that if $S$ is right coherent, then $(\mathcal{A}_C(R^{op}),l\mathcal{G}(\mathcal{F}_C(R)))$
is a coproduct-closed and product-closed duality pair and $\mathcal{A}_C(R^{op})$ is covering and preenveloping
in $\Mod R^{op}$, where $\mathcal{A}_C(R^{op})$ is the Auslander class in $\Mod R^{op}$
and $l\mathcal{G}(\mathcal{F}_C(R))$ is the left Gorenstein subcategory of $\Mod R$ relative to
$C$-flat modules.
\end{abstract}

\pagestyle{myheadings}
\markboth{\rightline {\scriptsize W. L. Song, T. W. Zhao, Z. Y. Huang}}
         {\leftline{\scriptsize  Duality Pairs Induced by One-Sided Gorenstein Subcategories}}

\section{Introduction} 

In relative homological algebra, the Gorenstein homological theory and the theory of covers and envelopes
have always been in the central, and their interplay has been researched extensively over the years.

In the Gorenstein homological theory, the properties of Gorenstein projective, injective and flat modules
and related modules are fundamental and important, see \cite{AB,EJ95,EJ00}.
For an abelian category $\A$ and an additive subcategory $\C$ of $\A$, as a common generalization
of Gorenstein projective and injective modules, Sather-Wagstaff, Sharif and White introduced in \cite{SSW} the
Gorenstein subcategory $\mathcal{G}(\C)$ of $\A$ relative to $\C$. It was proved that Gorenstein subcategories
unify many nice properties of Gorenstein projective modules and Gorenstein injective modules, see \cite{GD11,Hu,SSW}.
From the definition of the Gorenstein subcategory $\mathcal{G}(\C)$,
it is known that $\C$ should be a generator and a cogenerator for $\mathcal{G}(\C)$ simultaneously
and both functors $\Hom_\A(\C,-)$ and $\Hom_\A(-,\C)$ should possess certain exactness. These assumptions seem to be
strong to some extent. In \cite{SZ}, by modifying the definition of Gorenstein subcategories,
it was introduced the so-called right Gorenstein subcategories and left Gorenstein subcategories, such that
for a self-orthogonal subcategory $\C$ of $\A$, the Gorenstein subcategory $\mathcal{G}(\C)$
coincides with the intersection of the left and the right Gorenstein subcategories.

In the theory of covers and envelopes, given a subcategory, we always hope to know whether or when it
is (pre)covering or (pre)enveloping. This problem has been studied in depth,
see \cite{BR}, \cite{E1}--\cite{EO02}, \cite{GT12}--\cite{K}, \cite{X} and references therein.
Holm and J{\o}rgensen introduced the notion of duality pairs and proved the following remarkable result.
Let $R$ be an arbitrary ring and let $\mathscr{X}$ and $\mathscr{Y}$ be subcategories of $\Mod R$ and $\Mod R^{op}$ respectively.
If $(\mathscr{X},\mathscr{Y})$ is a duality pair, then the following assertions hold true:
(1) If $\mathscr{X}$ is closed under products, then $\mathscr{X}$ is preenveloping;
(2) if $\mathscr{X}$ is closed under coproducts, then $\mathscr{X}$ is covering; and
(3) if $_RR\in\mathscr{X}$ and $\mathscr{X}$ is closed under coproducts and extensions,
then $(\mathscr{X},\mathscr{X}^{\perp})$ is a perfect cotorsion pair
(\cite[Theorem 3.1]{HJ09}). By using it, they generalized a result in \cite{EH} about the covering and enveloping properties
of the Auslander and Bass classes in $\Mod R$ to the bounded derived category of $\Mod R$. Then, also by using this result of Holm and J{\o}rgensen,
Enochs and Iacob investigated in \cite{EI} the existence of Gorenstein injective envelopes over commutative noetherian rings.

Let $\C$ and $\mathscr{D}$ be subcategories of $\Mod R$ and $\Mod R^{op}$
respectively, and let $r\mathcal{G}(\C)$ and $l\mathcal{G}(\mathscr{D})$ be the corresponding right and left
Gorenstein subcategories respectively. In this paper, we will study when the pair $(r\mathcal{G}(\C),l\mathcal{G}(\mathscr{D}))$
is a duality pair in terms of the properties of $\C$ and $\mathscr{D}$. Then, combining with the result
of Holm and J{\o}rgensen mentioned above, we give some applications. The paper is organized as follows.

In Section 2, we give some terminology and notations.

Let $R$ be an arbitrary associative ring with identity and
$(-)^+:=\Hom_{\mathbb{Z}}(-,\mathbb{Q}/\mathbb{Z})$, where $\mathbb{Z}$ is the additive group of integers and $\mathbb{Q}$
is the additive group of rational numbers. For a class $\mathscr{X}$ of $R$-modules,
we write $\mathscr{X}^+:=\{X^+\mid X\in\mathscr{X}\}$.
In Section 3, we prove that if $\C$ and $\mathscr{D}$ are subcategories of $\Mod R$ and $\Mod R^{op}$
respectively satisfying the following conditions: (1) $\mathscr{D}^+\subseteq \C$ and $\C^+\subseteq\mathscr{D}$;
(2) $\C$ is preenveloping in $\Mod R$ and all modules in $\C$ are pure injective; and
(3) $\Ext_{R^{op}}^{\geq 1}(D,C^+)=0$ for any $D\in\mathscr{D}$ and $C\in\C$, then the pair $(r\mathcal{G}(\C),l\mathcal{G}(\mathscr{D}))$
is a duality pair. Furthermore, if $r\mathcal{G}(\C)$ is closed under coproducts (resp. products),
then $(r\mathcal{G}(\C),l\mathcal{G}(\mathscr{D}))$ is coproduct-closed (resp. product-closed)
and $r\mathcal{G}(\C)$ is covering (resp. preenveloping) in $\Mod R$ (Theorem \ref{3.5}).

In Section 4, we give two applications of Theorem \ref{3.5}.
Let $R,S$ be rings and $C$ a semidualizing ($R,S$)-bimodule. When $S$ is right coherent, we have
$$\mathcal{GF}_C(R)=r\mathcal{G}(\mathcal{F}_C(R)\cap\mathcal{PI}(R)),$$
where $\mathcal{GF}_C(R)$ is the subcategory of $\Mod R$ consisting of $C$-Gorenstein flat modules,
$\mathcal{F}_C(R)$ is the subcategory of $\Mod R$ consisting of $C$-flat modules and
$\mathcal{PI}(R)$ is the subcategory of $\Mod R$ consisting of pure injective modules (Theorem \ref{4.6}).
Then by using it, we give the first application of Theorem \ref{3.5} as follows: If $S$ is right coherent
and $C$ is faithfully semidualizing, then
$$(\mathcal{GF}_C(R),\mathcal{GI}_C(R^{op}))$$
is a coproduct-closed duality pair and $\mathcal{GF}_C(R)$ is covering in $\Mod R$, where
$\mathcal{G}\mathcal{I}_C(R^{op})$
is the subcategory of $\Mod R^{op}$ consisting of $C$-Gorenstein injective modules (Theorem \ref{4.8}).
As a consequence, we get that if $S$ is right coherent and ${_RC_S}$ is faithfully semidualizing,
then $(\mathcal{GF}_C(R),\mathcal{GF}_C(R)^{\bot})$ is a hereditary perfect cotorsion pair in $\Mod R$
(Corollary \ref{4.9}). It generalizes \cite[Theorem 2.12]{EJL}.

We observe that the Auslander class $\mathcal{A}_C(R^{op})$ coincides with the right Gorenstein subcategory
of $\Mod R^{op}$ relative to $C$-injective modules  (Lemma \ref{4.14}). According to this observation,
we apply Theorem \ref{3.5} to prove that if $S$ is right coherent, then
$$(\mathcal{A}_C(R^{op}),l\mathcal{G}(\mathcal{F}_C(R)))$$
is a coproduct-closed and product-closed duality pair and $\mathcal{A}_C(R^{op})$ is covering and preenveloping
in $\Mod R^{op}$, where $l\mathcal{G}(\mathcal{F}_C(R))$ is the left Gorenstein subcategory of $\Mod R$ relative to
$C$-flat modules (Theorem \ref{4.15}). Then, as a generalization of \cite[Theorem 3.11]{EH},
we get that if $S$ is right coherent, then
$$(\mathcal{A}_C(R^{op}),\mathcal{A}_C(R^{op})^{\bot})$$ is a hereditary perfect cotorsion pair
in $\Mod R^{op}$ (Corollary \ref{4.16}).

\section{Preliminaries}

In this section, $\A$ is an abelian category and all subcategories of $\A$ are full and closed under isomorphisms.
For a subcategory $\mathscr{X}$ of $\A$, we write
$${^\perp{\mathscr{X}}}:=\{A\in\A\mid\operatorname{Ext}^{\geq 1}_{\A}(A,X)=0 \mbox{ for any}\ X\in \mathscr{X}\},\
{{\mathscr{X}}^\perp}:=\{A\in\A\mid\operatorname{Ext}^{\geq 1}_{\A}(X,A)=0 \mbox{ for any}\ X\in \mathscr{X}\},$$
$${^{\perp_1}{\mathscr{X}}}:=\{A\in\A\mid\operatorname{Ext}^{1}_{\A}(A,X)=0 \mbox{ for any}\ X\in \mathscr{X}\},\
{{\mathscr{X}}^{\perp_1}}:=\{A\in\A\mid\operatorname{Ext}^{1}_{\A}(X,A)=0 \mbox{ for any}\ X\in \mathscr{X}\}.$$
For subcategories $\mathscr{X},\mathscr{Y}$ of $\A$, we write $\mathscr{X}\perp\mathscr{Y}$ if
$\operatorname{Ext}^{\geq 1}_{\A}(X,Y)=0$ for any $X\in \mathscr{X}$ and $Y\in \mathscr{Y}$.
Let $\C\subseteq \mathscr{X}$ be subcategories of $\A$. We say that $\C$ is a {\bf cogenerator} for $\mathscr{X}$
if for any $X\in \mathscr{X}$ there exists an exact sequence
$$0\ra X \ra C\ra X^{'}\ra 0$$ in $\A$ with $C\in \C$ and $X^{'}\in \mathscr{X}$; and we say that $\C$ is an
{\bf injective cogenerator} for $\mathscr{X}$ if $\C$ is a cogenerator for $\mathscr{X}$ and $\mathscr{X}\perp \C$.

A sequence $\mathbb{E}$ in $\A$ is called {\bf $\Hom_{\A}(\mathscr{X},-)$-exact}
(resp. {\bf $\Hom_{\A}(-,\mathscr{X})$-exact}) if
$\Hom_{\A}(X,\mathbb{E})$ (resp. $\Hom_{\A}(\mathbb{E},X)$) is exact for any $X\in \mathscr{X}$.
Following \cite{SSW}, we write
$\res \widetilde{{\mathcal{X}}}:=\{A\in\A\mid$ there exists a $\Hom_{\A}(\mathscr{X},-)$-exact exact sequence
$$\cdots\to X_i \to \cdots \to X_1 \to X_i \to A \to 0$$
in $\A$ with all $X_i$ in $\mathscr{X}\}$, and
$\cores \widetilde{\mathcal{X}}:=\{A\in\A\mid$ there exists a $\Hom_{\A}(-,\mathscr{X})$-exact exact sequence
$$0\to A\to X^0 \to X^1 \to \cdots \to X^i \to \cdots$$
in $\A$ with all $X^i$ in $\mathscr{X}\}$.

\begin{definition}\label{2.1}
{\rm (\cite{SSW})
Let $\C$ be a subcategory of $\A$. The {\bf Gorenstein subcategory}
$\mathcal{G}(\mathscr{C})$ of $\mathscr{A}$ (relative to $\C$) is defined as
$\{G\in\mathscr{A}\mid$ there exists a $\Hom_{\mathscr{A}}(\mathscr{C},-)$-exact and
$\Hom_{\mathscr{A}}(-,\mathscr{C})$-exact exact sequence
$$\cdots \to C_1 \to C_0 \to C^0 \to C^1 \to \cdots$$ in $\mathscr{A}$
with all $C_i,C^i$ in $\C$, such that $G\cong \Im(C_0\to C^0)\}$.}
\end{definition}

The Gorenstein subcategory unifies the following notions: modules of
Gorenstein dimension zero (\cite{AB}), Gorenstein projective modules,
Gorenstein injective modules (\cite{EJ95}), $V$-Gorenstein projective
modules, $V$-Gorenstein injective modules (\cite{EJL05}),
$\mathscr{W}$-Gorenstein modules (\cite{GD11}), and so on; see \cite{Hu} for the details.

Let $\C$ be an additive category of $\A$. Following \cite[Lemma 5.7]{Hu},
if $\C\bot\C$, then the Gorenstein subcategory
$$\mathcal{G}(\C)=({^\perp\mathscr{C}}\cap \cores\widetilde{\mathscr{C}})\cap
({\mathscr{C}^\perp}\cap \res\widetilde{\mathscr{C}}).$$
Motivated by this fact, it was introduced in \cite{SZ} the following

\begin{definition}\label{2.2}
We call $$r\mathcal{G}(\C):={^\perp\mathscr{C}}\cap \cores\widetilde{\mathscr{C}}\
(resp.\ l\mathcal{G}(\C):={\mathscr{C}^\perp}\cap \res\widetilde{\mathscr{C}})$$ the {\bf right}
(resp. {\bf left}) {\bf Gorenstein subcategory} of $\A$ (relative to $\C$).
\end{definition}

When $\C\bot\C$, we have $\mathcal{G}(\C)=r\mathcal{G}(\C)\cap l\mathcal{G}(\C)$.

\begin{definition} \label{2.3}
{\rm (\cite{E1,EJ00})
Let $\mathscr{X}\subseteq\mathscr{Y}$ be
subcategories of $\mathscr{A}$. The morphism $f: X\to Y$ in
$\mathscr{A}$ with $X\in\mathscr{X}$ and $Y\in \mathscr{Y}$ is called an {\bf $\mathscr{X}$-precover} of $Y$
if $\Hom_{\mathscr{A}}(X^{'},f)$ is epic for any $X^{'}\in\mathscr{X}$.
An {\bf $\mathscr{X}$-precover} $f: X\to Y$ is called an {\bf $\mathscr{X}$-cover}
if an endomorphism $h:X\to X$ is an automorphism whenever $f=fh$. The subcategory $\mathscr{X}$ is called {\bf (pre)covering}
in $\mathscr{Y}$ if any object in $\mathscr{Y}$ admits an $\mathscr{X}$-(pre)cover.
Dually, the notions of an {\bf $\mathscr{X}$-(pre)envelope} and a {\bf (pre)enveloping subcategory} are defined.}
\end{definition}

\begin{definition}\label{2.4}
{\rm (\cite{EJ00, EO02})
Let $\mathscr{U},\mathscr{V}$ be subcategories of $\A$.
\begin{enumerate}
\item[(1)] The pair $(\mathscr{U},\mathscr{V})$ is called a {\bf cotorsion pair} in $\A$ if
$\mathscr{U}={^{\bot_1}\mathscr{V}}$ and $\mathscr{V}={\mathscr{U}^{\bot_1}}$.
\item[(2)] A cotorsion pair $(\mathscr{U},\mathscr{V})$ is called {\bf perfect} if $\mathscr{U}$ is covering and $\mathscr{V}$
is enveloping in $\A$.
\item[(3)] Assume that $\A$ has enough projectives and enough injectives. A cotorsion pair $(\mathscr{U},\mathscr{V})$ is called
{\bf hereditary} if one of the following equivalent
conditions is satisfied.
\begin{enumerate}
\item[(3.1)] $\mathscr{U}\perp \mathscr{V}$.
\item[(3.2)] $\mathscr{U}$ is resolving in the sense that $\mathscr{U}$ contains all projectives
in $\A$, $\mathscr{U}$ is closed under extensions and kernels of epimorphisms.
\item[(3.3)] $\mathscr{V}$ is coresolving in the sense that $\mathscr{V}$ contains all injectives
in $\A$, $\mathscr{V}$ is closed under extensions and cokernels of monomorphisms.
\end{enumerate}
\end{enumerate}}
\end{definition}

\section{General results}

In this section, $R$ is an arbitrary associative ring with identity and $\Mod R$ is the category of left $R$-modules.
All subcategories of $\Mod R$ and $\Mod R^{op}$ are additive, full and closed under isomorphisms.

Recall that a short exact sequence in $\Mod R$ is called {\bf pure} if the functor $\Hom_R(M,-)$
preserves its exactness for any finitely presented left
$R$-module $M$; and a module $E\in \Mod R$ is called {\bf pure injective} if the functor $\Hom _R(-,E)$ preserves
the exactness of a short pure exact sequence in $\Mod R$ (cf. \cite{GT12,K}). We use $\mathcal{PI}(R)$ to denote
the subcategory of $\Mod R$ consisting of pure injective modules.

Let $\mathscr{D}$ be a subcategory of $\Mod R^{op}$. A sequence $\mathbb{E}$ in $\Mod R$ is called
{\bf $(\mathscr{D}\otimes_R-)$-exact} if $D\otimes_R\mathbb{E}$ is exact for any $D\in \mathscr{D}$. We write
$${\mathscr{D}^\top}:=\{M\in\Mod R\mid\Tor_{\geq 1}^R(D,M)=0 \mbox{ for any }D\in \mathscr{D}\}.$$

\begin{lemma}\label{3.1}
Let $\C$ and $\mathscr{D}$ be subcategories of $\Mod R$ and $\Mod R^{op}$
respectively, such that $\mathscr{D}^+\subseteq \C$. For a module $A\in\Mod R$, consider the following conditions.
\begin{enumerate}
\item[(1)] $A\in r\mathcal{G}(\C)$.
\item[(2)] $A\in {\mathscr{D}^\top}$ and there exists a $(\mathscr{D}\otimes_R-)$-exact exact sequence
$$0\ra A\ra C^0\ra C^1\ra \cdots\ra C^i\ra \cdots$$ in $\Mod R$ with all $C^i\in \C$.
\end{enumerate}
We have $(1)\Rightarrow (2)$. The converse holds true if $\C^+\subseteq\mathscr{D}$
and all modules in $\C$ are pure injective (that is, $\C\subseteq\mathcal{PI}(R)$).
\end{lemma}

\begin{proof}
$(1)\Rightarrow (2)$
Let $A\in r\mathcal{G}(\C)$. Then $A\in{^\bot\C}$ and there exists a $\Hom_{R}(-,\C)$-exact exact sequence
$$0 \rightarrow A \rightarrow C^0\rightarrow C^1\rightarrow \cdots\to C^i \to \cdots$$
in $\Mod R$ with all $C^i\in \C$. Let $D\in \mathscr{D}$. By assumption $D^+\in \C$ and hence the sequence
$$\cdots \to \Hom_{R}(C^i,D^+)\ra\cdots\ra \Hom_{R}(C^1,D^+)\ra \Hom_{R}(C^0,D^+)\ra \Hom_{R}(A,D^+)\ra 0$$ is exact.
By the adjoint isomorphism, we get the following exact sequence
$$\cdots\ra(D\otimes_RC^i)^+\ra \cdots\ra (D\otimes_RC^1)^+\ra (D\otimes_RC^0)^+\ra (D\otimes_RA)^+\ra 0,$$
which yields that the sequence
$$0\ra D\otimes_RA\ra D\otimes_RC^0\ra D\otimes_RC^1\ra\cdots\ra D\otimes_RC^i\ra\cdots$$
is also exact. On the other hand, by \cite[Lemma 2.16(b)]{GT12}
we have $[\Tor_i^R(D,A)]^+\cong \Ext^i_R(A,D^+)=0$ for any $i\geq 1$. So $\Tor_{\geq 1}^R(D,A)=0$ and $A\in {\mathscr{D}^\top}$.

Now assume that $\C^+\subseteq\mathscr{D}$ and all modules in $\C$ are pure injective. We will prove $(2)\Rightarrow (1)$.

Let $C\in\C$. Then $C^+\in \mathscr{D}$ and $C^{++}\in \C$. By \cite[Lemma 2.16(b)]{GT12} and (2),
we have $\Ext^i_R(A,C^{++})\cong[\Tor_i^R(C^+,A)]^+=0$ for any $i\geq 1$. By \cite[Theorem 2.27]{GT12}, we have that
$C$ is isomorphic to a direct summand of $C^{++}$. So $\Ext^{\geq 1}_R(A,C)=0$ and $A\in{^\perp\C}$.

By (2), we have the following exact sequence
$$0\ra C^+\otimes_RA\ra C^+\otimes_RC^0\ra C^+\otimes_RC^1\ra \cdots\ra C^+\otimes_RC^i\ra \cdots$$
with all $C^i\in \C$, which yields that the sequence
$$\cdots\ra (C^+\otimes_RC^i)^+\ra\cdots\ra (C^+\otimes_RC^1)^+\ra (C^+\otimes_RC^0)^+\ra (C^+\otimes_RA)^+\ra 0$$ is exact.
By the adjoint isomorphism, the sequence
$$\cdots\ra \Hom_{R}(C^i,C^{++})\ra\cdots\ra \Hom_{R}(C^1,C^{++})\ra \Hom_{R}(C^0,C^{++})\ra \Hom_{R}(A,C^{++})\ra 0$$ is also exact.
Because $C$ is isomorphic to a direct summand of $C^{++}$, we get the following exact sequence
$$\cdots\ra \Hom_{R}(C^i,C)\ra\cdots\ra \Hom_{R}(C^1,C)\ra \Hom_{R}(C^0,C)\ra \Hom_{R}(A,C)\ra 0.$$ Thus $A\in r\mathcal{G}(\C)$.
\end{proof}

By using Lemma \ref{3.1}, we have the following

\begin{proposition}\label{3.2}
Let $\C$ and $\mathscr{D}$ be subcategories of $\Mod R$ and $\Mod R^{op}$
respectively satisfying the following conditions.
\begin{enumerate}
\item[(a)] $\mathscr{D}^+\subseteq \C$ and $\C^+\subseteq\mathscr{D}$.
\item[(b)] $\C$ is preenveloping in $\Mod R$ and all modules in $\C$ are pure injective.
\item[(c)] $\mathscr{D}\bot\C^+$ (in particular, it is satisfied if $\mathscr{D}$ is self-orthogonal).
\end{enumerate}
Then the following are equivalent for any $A\in\Mod R$.
\begin{enumerate}
\item[(1)] $A\in r\mathcal{G}(\C)$.
\item[(2)] $A^+\in l\mathcal{G}(\mathscr{D})$.
\end{enumerate}
\end{proposition}

\begin{proof}
$(1)\Rightarrow (2)$ Let $A\in r\mathcal{G}(\C)$. Then by \cite[Lemma 2.16(b)]{GT12} and Lemma \ref{3.1},
$\Ext^i_R(D,A^+)\cong[\Tor_i^R(D,A)]^+=0$ for any $D\in\mathscr{D}$ and $i\geq 1$, and there exists a
$(\mathscr{D}\otimes_R-)$-exact exact sequence
$$0\ra A\ra C^0\ra C^1\ra \cdots\ra C^i\ra \cdots$$
in $\Mod R$ with all $C^i\in \C$. It induces an exact sequence
$$\cdots\ra {C^i}^+\ra\cdots\ra {C^1}^+\ra {C^0}^+\ra A^+\ra 0\eqno{(3.1)}$$
in $\Mod R^{op}$ with all ${C^i}^+\in\mathscr{D}$. On the other hand, since the sequence
$$0\ra D\otimes_RA\ra D\otimes_RC^0\ra D\otimes_RC^1\ra \cdots\ra D\otimes_RC^i\ra \cdots$$ is exact, the sequence
$$\cdots\ra (D\otimes_RC^i)^+\ra\cdots\ra (D\otimes_RC^1)^+\ra (D\otimes_RC^0)^+\ra (D\otimes_RA)^+\ra 0$$ is exact.
By the adjoint isomorphism, the sequence
$$\cdots\ra \Hom_R(D,{C^i}^+)\ra\cdots\ra \Hom_R(D,{C^1}^+)\ra \Hom_R(D,{C^0}^+)\ra \Hom_R(D,{A}^+)\ra 0$$ is also exact,
which shows that the exact sequence $(3.1)$ is $\Hom_R(\mathscr{D},-)$-exact. Thus $A^+\in l\mathcal{G}(\mathscr{D})$.

$(2)\Rightarrow (1)$ Let $A^+\in l\mathcal{G}(\mathscr{D})$. Then there exists an exact sequence
$$0\ra L\ra D_0\ra A^+\ra 0$$ in $\Mod R^{op}$ with $D_0\in\mathscr{D}$. It induces an exact sequence
$$0\ra A^{++}\ra D_0^+\ra L^+\ra 0$$ in $\Mod R$ with $D_0^+\in \C$. Then by \cite[Corollary 2.21]{GT12}, we get a monomorphism
$A\rightarrowtail D_0^+$ in $\Mod R$. Because $\C$ is preenveloping in $\Mod R$, we have a $\Hom_R(-,\C)$-exact exact sequence
$$0\ra A\buildrel {f^0} \over \longrightarrow C^0\ra A^1\ra 0\eqno{(3.2)}$$ in $\Mod R$ with $C^0\in\C$.
Because $\C^+\subseteq\mathscr{D}$ and $\mathscr{D}\bot\C^+$, we have ${C^0}^+\in\mathscr{D}\cap\mathscr{D}^{\bot}$,
and so ${C^0}^+\in l\mathcal{G}(\mathscr{D})$.

We claim that ${A^1}^+\in l\mathcal{G}(\mathscr{D})$. Firstly, we have an exact sequence
$$0\ra {A^1}^+\ra {C^0}^+\buildrel {f^0}^+ \over \longrightarrow A^+\ra 0\eqno{(3.3)}$$
in $\Mod R^{op}$ with ${C^0}^+\in \mathscr{D}$.
Let $D\in\mathscr{D}$. Then $D^+\in\C$ and $\Hom_R({f^0},D^{+})$ is epic. For any $X\in\Mod R$ and $Y\in\Mod R^{op}$,
it follows from the adjoint isomorphism theorem that there exist the following natural isomorphisms
$$\Hom_{R^{op}}(Y,{X}^+)\cong(Y\otimes_RX)^+\cong \Hom_R({X},Y^+).$$
So we have the following commutative diagram
\begin{gather*}
\begin{split}
\xymatrix{
& \Hom_{R^{op}}(D,{C^0}^+) \ar[rr]^{\Hom_{R^{op}}(D,{f^0}^{+})} \ar [d]^\cong  && \Hom_{R^{op}}(D, A^+) \ar [d]^\cong \\
& \Hom_R({C^0},D^+) \ar[rr]^{\Hom_R({f^0},D^{+})}  && \Hom_R(A,D^+),}
\end{split}
\end{gather*}
and hence $\Hom_{R^{op}}(D,{f^0}^{+})$ is epic. Because $\mathscr{D}\bot\C^+$
and the exact sequence (3.2) induces the following exact sequence
$$\Hom_{R^{op}}(D,{C^0}^+)\buildrel \Hom_{R^{op}}(D,{f^0}^{+}) \over \longrightarrow \Hom_{R^{op}}(D,A^+)
\ra\Ext^1_{R^{op}}(D,{A^1}^+)\ra\Ext^1_{R^{op}}(D,{C^0}^+)=0,$$ we have $\Ext^1_R(D,{A^1}^+)=0$.
Applying the dual of \cite[Proposition 3.4(2)]{SZ} to the exact sequence (3.3),
we have ${A^1}^+\in l\mathcal{G}(\mathscr{D})$. The claim is proved.

Similarly, we get a $\Hom_R(-,\C)$-exact exact sequence
$$0\ra A^1\buildrel {f^1} \over \longrightarrow C^1\ra A^2\ra 0$$
in $\Mod R$ with $C^1\in\C$ and ${A^2}^+\in l\mathcal{G}(\mathscr{D})$.
Continuing this process, we get a $\Hom_R(-,\C)$-exact exact sequence
$$0\ra A\buildrel {f^0} \over \longrightarrow C^0\buildrel {f^1} \over \longrightarrow C^1
\ra \cdots\buildrel {f^i} \over \longrightarrow C^i\to \cdots$$
in $\Mod R$ with all $C^i$ in $\C$. On the other hand, because
$A^+\in l\mathcal{G}(\mathscr{D})$, we have $[\Tor_i^R(D,A)]^+\cong\Ext^i_R(D,A^+)=0$ for any $i\geq 1$
by \cite[Lemma 2.16(b)]{GT12}. So $\Tor_{\geq 1}^R(D,A)=0$ and $A\in{\mathscr{D}^\top}$. It follows from
Lemma \ref{3.1} that $A\in r\mathcal{G}(\C)$.
\end{proof}

The following is the definition of duality pairs (cf. \cite{EI,HJ09}).

\begin{definition}\label{3.3}
{\rm Let $\mathscr{X}$ and $\mathscr{Y}$ be subcategories of $\Mod R$ and $\Mod R^{op}$ respectively.
\begin{enumerate}
\item[(1)] The pair ($\mathscr{X},\mathscr{Y}$) is called a {\bf duality pair} if the following conditions are satisfied.
\begin{enumerate}
\item[(1.1)] For a module $X\in\Mod R$, $X\in\mathscr{X}$  if and only if $X^{+}\in \mathscr{Y}$.
\item[(1.2)] $\mathscr{Y}$ is closed under direct summands and finite direct sums.
\end{enumerate}
\item[(2)] A duality pair ($\mathscr{X},\mathscr{Y}$) is called {\bf (co)product-closed} if $\mathscr{X}$ is closed under
(co)products.
\item[(3)] A duality pair ($\mathscr{X},\mathscr{Y}$) is called {\bf perfect} if it is coproduct-closed,
$_RR\in\mathscr{X}$ and $\mathscr{X}$ is closed under extensions.
\end{enumerate}}
\end{definition}

We also recall the following remarkable result.

\begin{lemma}\label{3.4}
{\rm (\cite[p.7, Theorem]{EI} and \cite[Theorem 3.1]{HJ09})}
Let $\mathscr{X}$ and $\mathscr{Y}$ be subcategories of $\Mod R$ and $\Mod R^{op}$ respectively.
If $(\mathscr{X},\mathscr{Y})$ is a duality pair, then the following assertions hold true.
\begin{enumerate}
\item[(1)] If $(\mathscr{X},\mathscr{Y})$ is coproduct-closed, then $\mathscr{X}$ is covering.
\item[(2)] If $(\mathscr{X},\mathscr{Y})$ is product-closed, then $\mathscr{X}$ is preenveloping.
\item[(3)] If $(\mathscr{X},\mathscr{Y})$ is perfect, then $(\mathscr{X},\mathscr{X}^{\perp})$ is a perfect cotorsion pair.
\end{enumerate}
\end{lemma}

Now we are in a position to give the following

\begin{theorem}\label{3.5}
Under the assumptions in Proposition \ref{3.2}, the pair $$(r\mathcal{G}(\C),l\mathcal{G}(\mathscr{D}))$$
is a duality pair. Furthermore, if $r\mathcal{G}(\C)$ is closed under coproducts (resp. products),
then this duality pair is coproduct-closed (resp. product-closed)
and $r\mathcal{G}(\C)$ is covering (resp. preenveloping) in $\Mod R$.
\end{theorem}

\begin{proof}
Because $\mathscr{D}$ is additive, $l\mathcal{G}(\mathscr{D})$ is closed under finite direct sums.
By \cite[Theorem 4.6(1)]{Hu}, $l\mathcal{G}(\mathscr{D})$ is closed under direct summands.
So $(r\mathcal{G}(\C),l\mathcal{G}(\mathscr{D}))$ is a duality pair by Propositions \ref{3.2}.
If $r\mathcal{G}(\C)$ is closed under coproducts (resp. products), then $(r\mathcal{G}(\C),l\mathcal{G}(\mathscr{D}))$
is coproduct-closed (resp. product-closed). It follows from Lemma \ref{3.4} that $r\mathcal{G}(\C)$
is covering (resp. preenveloping) in $\Mod R$.
\end{proof}

\section{Applications}

\subsection{The pair $(\mathcal{GF}_C(R),\mathcal{GI}_C(R^{op}))$}
In this subsection, we will give the first application of Theorem \ref{3.5}.

\begin{definition} \label{4.1}
{\rm (\cite{HW07}).
Let $R$ and $S$ be rings. An ($R,S$)-bimodule $_RC_S$ is called
{\bf semidualizing} if the following conditions are satisfied.
\begin{enumerate}
\item[(a1)] $_RC$ admits a degreewise finite $R$-projective resolution.
\item[(a2)] $C_S$ admits a degreewise finite $S$-projective resolution.
\item[(b1)] The homothety map $_RR_R\stackrel{_R\gamma}{\rightarrow} \Hom_{S^{op}}(C,C)$ is an isomorphism.
\item[(b2)] The homothety map $_SS_S\stackrel{\gamma_S}{\rightarrow} \Hom_{R}(C,C)$ is an isomorphism.
\item[(c1)] $\Ext_{R}^{\geqslant 1}(C,C)=0$.
\item[(c2)] $\Ext_{S^{op}}^{\geqslant 1}(C,C)=0$.
\end{enumerate}
A semidualizing bimodule $_RC_S$ is called {\bf faithful} if the following conditions are satisfied.
\begin{enumerate}
\item[(f1)] If $M\in \Mod R$ and $\Hom_R(C,M)=0$, then $M=0$.
\item[(f2)] If $N\in \Mod S^{op}$ and $\Hom_{S^{op}}(C,N)=0$, then $N=0$.
\end{enumerate}}
\end{definition}

Wakamatsu in \cite{W1} introduced and studied the so-called {\bf generalized tilting modules},
which are usually called {\bf Wakamatsu tilting modules}, see \cite{BR, MR}. Note that
a bimodule $_RC_S$ is semidualizing if and only if it is Wakamatsu tilting (\cite[Corollary 3.2]{W3}).
Examples of semidualizing bimodules are referred to \cite{HW07,W2}. In particular,
$_RR_R$ is a faithfully semidualizing bimodule; and all semidualizing modules over commutative rings are faithful
(\cite[Proposition 3.1]{HW07}).

From now on, $R$ and $S$ are arbitrary rings and we fix a semidualizing bimodule $_RC_S$.
We write $$\mathcal{P}_C(R):=\{C\otimes_SP\mid P\ {\rm \ is\ projective\ in}\ \Mod S\},$$
$$\mathcal{F}_C(R):=\{C\otimes_SP\mid P\ {\rm \ is\ flat\ in}\ \Mod S\},$$
$$\mathcal{I}_C(R^{op}):=\{\Hom_{S^{op}}(C,I)\mid I\ {\rm \ is\ injective\ in}\ \Mod S^{op}\}.$$
The modules in $\mathcal{P}_C(R)$, $\mathcal{F}_C(R)$ and $\mathcal{I}_C(R^{op})$ are called {\bf $C$-projective},
{\bf $C$-flat} and {\bf $C$-injective}, respectively.
When ${_RC_S}={_RR_R}$, $C$-projective, $C$-flat and $C$-injective modules are exactly
projective, flat and injective modules, respectively.


\begin{lemma}\label{4.2}
\begin{enumerate}
\item[]
\item[(1)] For a module $A\in\Mod R$, if $A\in\mathcal{F}_C(R)$, then $A^+\in\mathcal{I}_C(R^{op})$.
The converse holds true if $S$ is a right coherent ring and ${_RC_S}$ is faithfully semidualizing.
\item[(2)] Assume that $S$ is a right coherent ring. For a module $B\in\Mod R^{op}$, if $B\in\mathcal{I}_C(R^{op})$,
then $B^+\in\mathcal{F}_C(R)$. The converse holds true if $S$ is a right noetherian ring and ${_RC_S}$ is faithfully semidualizing.
\item[(3)] If $R$ is a right coherent ring, then $\mathcal{F}_C(R)^+\subseteq\mathcal{I}_C(R^{op})$
and $\mathcal{I}_C(R^{op})^+\subseteq \mathcal{F}_C(R)\cap\mathcal{PI}(R)$.
\end{enumerate}
\end{lemma}

\begin{proof}
(1) The first assertion follows from \cite[Lemma 2.3(1)]{TH3}.
Conversely, let $S$ be a right coherent ring and ${_RC_S}$ faithfully semidualizing,
and let $A\in\Mod R$ such that $A^+\in\mathcal{I}_C(R^{op})$. Then $A^+=\Hom_{S^{op}}(C,I)$ for
some injective module $I$ in $\Mod S^{op}$. By \cite[Lemma 2.16(c)]{GT12}, we have
$$A^{++}=[\Hom_{S^{op}}(C,I)]^+\cong C\otimes_SI^+.$$
By \cite[Corollary 2.18(b)]{GT12}, we have that $I^+\in\Mod S$ is flat, and so $A^{++}\in\mathcal{F}_C(R)$.
By \cite[Theorem 2.27]{GT12}, we have that $A$ is isomorphic to a pure submodule of $A^{++}$.
It follows from \cite[Lemma 5.2(a)]{HW07} that $A\in\mathcal{F}_C(R)$.

(2) The first assertion follows from \cite[Lemma 2.3(2)]{TH3}.
Conversely, let $S$ be a right noetherian ring and ${_RC_S}$ is faithfully semidualizing,
and let $B\in\Mod R^{op}$ such that $B^+\in\mathcal{F}_C(R)$. Then $B^+=C\otimes_SF$
for some flat module $F$ in $\Mod S$. By the adjoint isomorphism theorem, we have
$$B^{++}=(C\otimes_SF)^+\cong\Hom_{S^{op}}(C,F^+).$$
By \cite[Corollary 2.18(b)]{GT12}, we have that $F^+\in\Mod S^{op}$ is injective, and so $B^{++}\in\mathcal{I}_C(R^{op})$.
By \cite[Theorem 2.27]{GT12}, we have that $B$ is isomorphic to a pure submodule of $B^{++}$.
It follows from \cite[Lemma 5.2(b)]{HW07} that $B\in\mathcal{I}_C(R^{op})$.

(3) Because $\mathcal{I}_C(R^{op})^+\subseteq\mathcal{PI}(R)$ by \cite[Theorem 2.27]{GT12},
the assertion follows from the above two assertions.
\end{proof}

We introduce the notion of $C$-cotorsion modules as follows, which is a $C$-version of that of cotorsion modules
in \cite{E2,EJ00,X}.

\begin{definition} \label{4.3}
{\rm A module $A\in \Mod R$ is called {\bf $C$-cotorsion} if $A\in{\mathcal{F}_C(R)^\bot}$.}
\end{definition}

We use $\mathcal{C}_C(R)$ to denote the subcategory of $\Mod R$ consisting of $C$-cotorsion modules.
For a module $A\in\Mod R$, $\sigma_A:A\rightarrow A^{++}$
defined by $\sigma_A(x)(f)=f(x)$ for any $x\in A$ and $f\in A^+$ is the canonical valuation homomorphism.
The first assertion in the following result is a $C$-version of \cite[Lemma 3.2.3]{X} (also cf. \cite[Lemma 2.3]{E2}).

\begin{proposition} \label{4.4}
Let $S$ be a right coherent ring. Then we have
\begin{enumerate}
\item[(1)] $\mathcal{F}_C(R)\cap\mathcal{C}_C(R)=\mathcal{F}_C(R)\cap\mathcal{PI}(R)$.
\item[(2)] $\mathcal{F}_C(R)\cap\mathcal{PI}(R)$ is an injective cogenerator for $\mathcal{F}_C(R)$.
\item[(3)] If $C$ is faithfully semidualizing, then $\mathcal{F}_C(R)\cap\mathcal{PI}(R)$ is preenveloping in $\Mod R$.
\end{enumerate}
\end{proposition}

\begin{proof}
(1) Let $A\in\mathcal{F}_C(R)\cap\mathcal{C}_C(R)$. Then $A^+\in\mathcal{I}_C(R^{op})$ and $A^{++}\in\mathcal{F}_C(R)$
by Lemma \ref{4.2}. By \cite[Corollary 2.21(b)]{GT12}, we have the following pure exact sequence
$$0\to A \buildrel {\sigma_A}\over \longrightarrow A^{++}\to A^{++}/A\to 0\eqno{(4.1)}$$
in $\Mod R$. Then $A^{++}/A\in \mathcal{F}_C(R)$ by \cite[Lemma 5.2(a)]{HW07}. So this exact sequence splits,
and hence $A$ is isomorphic to a direct summand of $A^{++}$. Because $A^{++}$
is pure injective by \cite[Theorem 2.27]{GT12}, we have that
$A$ is also pure injective and $A\in\mathcal{F}_C(R)\cap\mathcal{PI}(R)$.

Conversely, let $A\in\mathcal{F}_C(R)\cap\mathcal{PI}(R)$ and $Q\in\mathcal{F}_C(R)$.
Then $Q=C\otimes_SF$ for some flat module $F$ in $\Mod S$. For any $i\geq 1$, we have
\begin{align*}
&\ \ \ \ \ \  \ \ \ \Ext_R^i(Q,A^{++})\\
&\ \ \ \ \ =\Ext_R^i(C\otimes_SF,A^{++})\\
&\ \ \ \ \ \cong[\Tor_i^R(A^+,C\otimes_SF)]^+\ \text{(by \cite[Lemma 2.16(b)]{GT12})}\\
&\ \ \ \ \ \cong[\Tor_i^R(A^+,C)\otimes_SF]^+\ \text{(by \cite[Theorem 9.48]{R})}\\
&\ \ \ \ \ \cong[(\Ext_R^i(C,A))^+\otimes_SF]^+\ \text{(by \cite[Lemma 2.16(d)]{GT12})}\\
&\ \ \ \ \ =0 \ \text{(by \cite[Lemma 2.5(a)]{TH3})}.
\end{align*}
Because $A$ is isomorphic to a direct summand of $A^{++}$ by \cite[Theorem 2.27]{GT12},
we have that $\Ext_R^i(Q,A)=0$ for any $i\geq 1$ and $A$ is $C$-cotorsion. Thus
$A\in\mathcal{F}_C(R)\cap\mathcal{C}_C(R)$.

(2) Let $A\in\mathcal{F}_C(R)$. Then we have an exact sequence as (4.1) with $A^{++}\in\mathcal{F}_C(R)\cap\mathcal{PI}(R)$
and $A^{++}/A\in \mathcal{F}_C(R)$. By (1), we have $A^{++}\in\mathcal{C}_C(R)(={\mathcal{F}_C(R)^\bot})$. The assertion follows.

(3) Let $A\in\Mod R$. By \cite[Proposition 5.3(d)]{HW07}, we have an $\mathcal{F}_C(R)$-preenvelope $f:A\to Q$ of $A$.
Let $Q^{'}\in\mathcal{F}_C(R)\cap\mathcal{PI}(R)$ and $f^{'}\in\Hom_R(A,Q^{'})$. Then there exists $h\in \Hom_R(Q,Q^{'})$ such that
$f^{'}=hf$. From the following commutative diagram
\begin{gather*}
\begin{split}
\xymatrix{
& Q \ar[rr]^{\sigma_Q} \ar [d]^{h}  && Q^{++} \ar [d]^{h^{++}} \\
& Q^{'} \ar[rr]^{\sigma_{Q^{'}}}  &&{Q^{'}}^{++},}
\end{split}
\end{gather*}
we have $\sigma_{Q^{'}}h=h^{++}\sigma_Q$. By \cite[Theorem 2.27]{GT12}, $\sigma_{Q^{'}}$ is a split monomorphism
and there exists $g\in\Hom_R({Q^{'}}^{++},Q^{'})$ such that $g\sigma_{Q^{'}}=1_{Q^{'}}$. So we have
$$f^{'}=hf=(g\sigma_{Q^{'}})hf=(gh^{++})(\sigma_Qf).$$
It follows that the homomorphism $\sigma_Qf:A\to Q^{++}$ is an $\mathcal{F}_C(R)\cap\mathcal{PI}(R)$-preenvelope of $A$.
\end{proof}

The following notions were introduced by Holm and J$\phi$gensen in \cite{HJ06} for
commutative rings. The following are their non-commutative versions.

\begin{definition} \label{4.5}
{\rm \begin{enumerate}
\item[]
\item[(1)] A module $M\in\Mod R$ is called {\bf $C$-Gorenstein projective} if
$M\in{^{\bot}\mathcal{P}_C(R)}$ and there exists a $\Hom_R(-,\mathcal{P}_C(R))$-exact exact sequence
$$0\rightarrow M\rightarrow G^0\rightarrow G^1\rightarrow \cdots\rightarrow G^i\rightarrow \cdots$$
in $\Mod R$ with all $G^i$ in $\mathcal{P}_C(R)$.
\item[(2)] A module $M\in\Mod R$ is called {\bf $C$-Gorenstein flat} if
$M\in\mathcal{I}_C(R^{op})^{\top}$ and there exists an $(\mathcal{I}_C(R^{op})\otimes_R-)$-exact exact sequence
$$0\rightarrow M\rightarrow Q^0\rightarrow Q^1\rightarrow \cdots\rightarrow Q^i\rightarrow \cdots$$
in $\Mod R$ with all $Q^i$ in $\mathcal{F}_C(R)$.
\item[(3)] A module $N\in\Mod R^{op}$ is called {\bf $C$-Gorenstein injective} if
$N\in\mathcal{I}_C(R^{op})^{\bot}$ and there exists a $\Hom_{R^{op}}(\mathcal{I}_C(R^{op}),-)$-exact exact sequence
$$\cdots\rightarrow E_i\rightarrow \cdots\rightarrow E_1\rightarrow E_0\rightarrow N\rightarrow 0$$
in $\Mod R^{op}$ with all $E_i$ in $\mathcal{I}_C(R^{op})$.
\end{enumerate}}
\end{definition}

We use $\mathcal{GP}_C(R)$ (resp. $\mathcal{GF}_C(R)$) to denote the subcategory of $\Mod R$
consisting of $C$-Gorenstein projective (resp. flat) modules, and use $\mathcal{GI}_C(R^{op})$
to denote the subcategory of $\Mod R^{op}$ consisting of $C$-Gorenstein injective modules.
When $_RC_S={_RR_R}$, $C$-Gorenstein projective, flat and injective modules are
the classical Gorenstein projective, flat and injective modules, respectively (\cite{EJ00,Ho04}).

The following equivalent characterizations of $C$-Gorenstein flat modules are useful in the sequel.

\begin{theorem} \label{4.6}
Let $S$ be a right coherent ring and $A\in\Mod R$. Then the following statements are equivalent.
\begin{enumerate}
\item[(1)] $A\in\mathcal{GF}_C(R)$.
\item[(2)] $A\in r\mathcal{G}(\mathcal{F}_C(R)\cap\mathcal{C}_C(R))$.
\item[(3)] $A\in{^\bot(\mathcal{F}_C(R)\cap\mathcal{C}_C(R))}$ and there exists a
$\Hom(-,\mathcal{F}_C(R)\cap\mathcal{C}_C(R))$-exact exact sequence
$$0\ra A\ra Q^0\ra Q^1\ra \cdots\ra Q^i\ra \cdots$$ in $\Mod R$ with all $Q^i\in\mathcal{F}_C(R)$.
\item[(4)] $A\in r\mathcal{G}(\mathcal{F}_C(R)\cap\mathcal{PI}(R))$.
\item[(5)] $A\in{^\bot(\mathcal{F}_C(R)\cap\mathcal{PI}(R))}$ and there exists a
$\Hom(-,\mathcal{F}_C(R)\cap\mathcal{PI}(R))$-exact exact sequence
$$0\ra A\ra Q^0\ra Q^1\ra \cdots\ra Q^i\ra \cdots$$ in $\Mod R$ with all $Q^i\in\mathcal{F}_C(R)$.
\end{enumerate}
\end{theorem}

\begin{proof}
By Proposition \ref{4.4}, we have that
$\mathcal{F}_C(R)\cap\mathcal{PI}(R)(=\mathcal{F}_C(R)\cap\mathcal{C}_C(R))$ is an
injective cogenerator for $\mathcal{F}_C(R)$.
Then by \cite[Theorem 3.6]{SZ}, we have $(3)\Leftrightarrow (2)\Leftrightarrow (4)\Leftrightarrow (5)$.
By Lemma \ref{4.2}(3), we have $(\mathcal{F}_C(R)\cap\mathcal{PI}(R))^+\subseteq\mathcal{I}_C(R^{op})$
and $\mathcal{I}_C(R^{op})^+\subseteq \mathcal{F}_C(R)\cap\mathcal{PI}(R)$.
Then by Lemma \ref{3.1}, we have $(1)\Leftrightarrow (4)$.
\end{proof}

Recall from \cite{K} that a subcategory of $\Mod S$ is called {\bf definable}
if it is closed under direct limits, products and pure submodules in $\Mod S$.
We have the following

\begin{lemma} \label{4.7}
\begin{enumerate}
\item[]
\item[(1)] $\mathcal{G}\mathcal{P}_C(R)=r\mathcal{G}(\mathcal{P}_C(R))$
and $\mathcal{G}\mathcal{I}_C(R^{op})=l\mathcal{G}(\mathcal{I}_C(R^{op}))$.
\item[(2)] If $S$ is a right coherent and left perfect ring, then
$\mathcal{GP}_C(R)=\mathcal{GF}_C(R)$.
\end{enumerate}
\end{lemma}

\begin{proof}
(1) It is trivial.

(2) Let $S$ be a right coherent and left perfect ring. Then any flat module in $\Mod S$ is projective
by \cite[Theorem 28.4]{AF}, and so $\mathcal{P}_C(R)=\mathcal{F}_C(R)$. On the other hand,
by \cite[Theorem 27.11]{AF}, any projective module in $\Mod S$ has a decomposition as a direct sum of
indecomposable projective submodules. It follows from \cite[Theorem 5]{SE}, the subcategory of $\Mod S$ consisting of
projective modules is definable. So all projective modules in $\Mod S$ are pure injective by \cite[Corollary 2.7]{K},
and hence all modules in $\mathcal{P}_C(R)$ are pure injective by \cite[Theorem 3.5(a)]{YH}.
Thus $\mathcal{P}_C(R)=\mathcal{F}_C(R)\cap\mathcal{PI}(R)$. Now by (1) and Theorem \ref{4.6}, we have
$$\mathcal{G}\mathcal{P}_C(R)=r\mathcal{G}(\mathcal{P}_C(R))=r\mathcal{G}(\mathcal{F}_C(R)\cap\mathcal{PI}(R))=\mathcal{GF}_C(R).$$
\end{proof}

Now we give an application of Theorem \ref{3.5} as follows.

\begin{theorem}\label{4.8}
If $S$ is a right coherent ring and ${_RC_S}$ is faithfully semidualizing,
then $$(\mathcal{GF}_C(R),\mathcal{GI}_C(R^{op}))$$ is a coproduct-closed duality pair
and $\mathcal{GF}_C(R)$ is covering in $\Mod R$.
\end{theorem}

\begin{proof}
We have
\begin{enumerate}
\item[(a)] $(\mathcal{F}_C(R)\cap\mathcal{PI}(R))^+\subseteq\mathcal{I}_C(R^{op})$
and $\mathcal{I}_C(R^{op})^+\subseteq \mathcal{F}_C(R)\cap\mathcal{PI}(R)$ by Lemma \ref{4.2}(3).
\item[(b)] By Proposition \ref{4.4}(3), $\mathcal{F}_C(R)\cap\mathcal{PI}(R)$ is preenveloping in $\Mod R$.
It is trivial that all modules in $\mathcal{F}_C(R)\cap\mathcal{PI}(R)$ are pure injective.
\item[(c)] By \cite[Lemma 2.5(b)]{TH3}, we have that $\mathcal{I}_C(R^{op})$ is self-orthogonal.
\end{enumerate}
Note that $\mathcal{GF}_C(R)=r\mathcal{G}(\mathcal{F}_C(R)\cap\mathcal{PI}(R))$
and $\mathcal{G}\mathcal{I}_C(R^{op})=l\mathcal{G}(\mathcal{I}_C(R^{op}))$ by Theorem \ref{4.6}
and Lemma \ref{4.7}(1). It follows from Theorem \ref{3.5} that
$$(\mathcal{GF}_C(R),\mathcal{GI}_C(R^{op}))
(=(r\mathcal{G}(\mathcal{F}_C(R)\cap\mathcal{PI}(R)),l\mathcal{G}(\mathcal{I}_C(R^{op})))$$
is a duality pair. Moreover, by \cite[Proposition 5.1(a)]{HW07}, we have that $\mathcal{F}_C(R)$ is closed under coproducts.
Because tensor products commute with coproducts, from the definition of $C$-Gorenstein flat modules it is easy to get that
$\mathcal{GF}_C(R)$ is closed under coproducts. Thus the above duality pair
is a coproduct-closed and $\mathcal{GF}_C(R)$ is covering in $\Mod R$ by Theorem \ref{3.5} again.
\end{proof}

The following corollary was proved in \cite[Theorem 2.12]{EJL} for the case $_RC_S={_RR_R}$.

\begin{corollary}\label{4.9}
If $S$ is a right coherent ring and ${_RC_S}$ is faithfully semidualizing,
then $$(\mathcal{GF}_C(R),\mathcal{GF}_C(R)^{\bot})$$ is a hereditary perfect cotorsion pair in $\Mod R$.
\end{corollary}

\begin{proof}
By using an argument similar to that in the proof of \cite[Theorem 3.7]{Ho04}, we get that
$\mathcal{GF}_C(R)$ is projectively resolving.
Now the assertion follows from Theorem \ref{4.8} and Lemma \ref{3.4}(3).
\end{proof}

By Theorem \ref{4.8} and Lemma \ref{4.7}(2), we have

\begin{corollary}\label{4.10}
If $S$ is a right coherent and left perfect ring and ${_RC_S}$ is faithfully semidualizing,
then $$(\mathcal{GP}_C(R),\mathcal{GI}_C(R^{op}))$$ is a coproduct-closed duality pair
and $\mathcal{GP}_C(R)$ is covering in $\Mod R$.
\end{corollary}

Note that for a commutative noetherian ring $R$ and a semidualizing bimodule $_RC_R$,
it was proved in \cite[Theorem 3.11]{HGD15} that $R$ is artinian if and only if
$(\mathcal{GP}_C(R),\mathcal{GI}_C(R))$ is a (coproduct-closed) duality pair.

We use $\mathcal{GF}(R)$
(resp. $\mathcal{GI}(R^{op})$) to the subcategory of $\Mod R$ (resp. $\Mod R^{op}$) consisting
Gorenstein flat (resp. injective) modules. By Theorem \ref{4.8}, we have

\begin{corollary}\label{4.11}
If $R$ is a right coherent ring,
then $$(\mathcal{GF}(R),\mathcal{GI}(R^{op}))$$ is a coproduct-closed duality pair
and $\mathcal{GF}(R)$ is covering in $\Mod R$.
\end{corollary}

\subsection{The pair $(\mathcal{A}_C(R^{op}),l\mathcal{G}(\mathcal{F}_C(R)))$}

In this subsection, we will give the second application of Theorem \ref{3.5}. We write
$${^{\top}{_RC}}:=\{A\in\Mod R^{op}\mid\Tor_{\geq 1}^R(A,C)=0\}\ \text{and}\
{{C_S}^{\top}}:=\{B\in\Mod S\mid\Tor_{\geq 1}^S(C,B)=0\}.$$

\begin{definition} \label{4.12}
{\rm (\cite{HW07})
\begin{enumerate}
\item[(1)] The {\bf Auslander class} $\mathcal{A}_{C}(R^{op})$ with respect to $C$ consists of all modules $N$
in $\Mod R^{op}$ satisfying the following conditions.
\begin{enumerate}
\item[(a1)] $N\in{^{\top}{_RC}}$.
\item[(a2)] $N\otimes _{R}C\in{{C_S}^{\perp}}$.
\item[(a3)] The canonical valuation homomorphism
$$\mu_N:N\rightarrow \Hom_{S^{op}}(C,N\otimes_RC)$$
defined by $\mu_N(x)(c)=x\otimes c$ for any $x\in N$ and $c\in C$ is an isomorphism in $\Mod R^{op}$.
\end{enumerate}
\item[(2)] The {\bf Bass class} $\mathcal{B}_C(R)$ with respect to $C$ consists of all modules $M$
in $\Mod R$ satisfying the following conditions.
\begin{enumerate}
\item[(b1)] $M\in{_RC^{\perp}}$.
\item[(b2)] $\Hom_R(C,M)\in{{C_S}^{\top}}$.
\item[(b3)] The canonical valuation homomorphism
$$\theta_M:C\otimes_S\Hom_R(C,M)\rightarrow M$$
defined by $\theta_M(c\otimes f)=f(c)$ for any $c\in C$ and $f\in \Hom_R(C,M)$ is an isomorphism in $\Mod R$.
\end{enumerate}
\end{enumerate}}
\end{definition}

The following two observations are useful.

\begin{lemma}\label{4.13}
$\mathcal{F}_C(R)\bot[\mathcal{I}_C(R^{op})]^+$.
\end{lemma}

\begin{proof}
Let $A\in\mathcal{F}_C(R)$ and $B\in[\mathcal{I}_C(R^{op})]^+$. Then $A=C\otimes_SF$ and $B=[\Hom_{S^{op}}(C,I)]^+$
for some flat module $F$ in $\Mod S$ and injective module $I$ in $\Mod S^{op}$. For any $i\geq 1$, we have
\begin{align*}
&\ \ \ \ \ \  \ \ \ \Ext_R^i(A,B)\\
&\ \ \ \ \ =\Ext_R^i(C\otimes_SF,[\Hom_{S^{op}}(C,I)]^+)\\
&\ \ \ \ \ \cong[\Tor_i^R(\Hom_{S^{op}}(C,I),C\otimes_SF)]^+\ \text{(by \cite[Lemma 2.16(b)]{GT12})}\\
&\ \ \ \ \ \cong[\Tor_i^S(I,F)]^+\ \text{(by \cite[Theorem 6.4(c)]{HW07})}\\
&\ \ \ \ \ =0.
\end{align*}
\end{proof}

We use $\Prod C^+$ to denote the subcategory of $\Mod R^{op}$ consisting of modules consisting of
direct summands of products of copies of $C^+$.

\begin{lemma}\label{4.14}
$\mathcal{A}_C(R^{op})=r\mathcal{G}(\mathcal{I}_C(R^{op}))$ and $\mathcal{B}_C(R)=l\mathcal{G}(\mathcal{P}_C(R))$.
\end{lemma}

\begin{proof}
By \cite[Proposition 2.4(2)]{LHX13}, we have $\mathcal{I}_C(R^{op})=\Prod C^+$.
Then by \cite[Lemma 2.16(b)]{GT12}, it is easy to get
$${^{\top}{_RC}}={^{\bot}(C^+)}={^{\bot}(\Prod C^+)}={^{\bot}\mathcal{I}_C(R^{op})}.$$
So $\mathcal{A}_C(R^{op})=r\mathcal{G}(\mathcal{I}_C(R^{op}))$ by \cite[Theorem 3.11(1)]{TH2}.
On the other hand, by \cite[Theorem 3.9]{TH1}, we have $\mathcal{B}_C(R)=l\mathcal{G}(\mathcal{P}_C(R))$.
\end{proof}

We are ready to prove the following

\begin{theorem}\label{4.15}
If $S$ is a right coherent ring,
then $$(\mathcal{A}_C(R^{op}),l\mathcal{G}(\mathcal{F}_C(R)))$$ is a coproduct-closed and product-closed duality pair
and $\mathcal{A}_C(R^{op})$ is covering and preenveloping in $\Mod R^{op}$.
\end{theorem}

\begin{proof}
We have
\begin{enumerate}
\item[(a)] $[\mathcal{I}_C(R^{op})]^+\subseteq \mathcal{F}_C(R)$ and
$[\mathcal{F}_C(R)]^+\subseteq \mathcal{I}_C(R^{op})$ by Lemma \ref{4.2}.
\item[(b)] By \cite[Proposition 5.3(c)]{HW07}, $\mathcal{I}_C(R^{op})$ is preenveloping in $\Mod R^{op}$.
Because $\mathcal{I}_C(R^{op})=\Prod C^+$ by \cite[Proposition 2.4(2)]{LHX13},
it follows from \cite[Theorem 2.27]{GT12} that all modules in $\mathcal{I}_C(R^{op})$ are pure injective.
\item[(c)] By Lemma \ref{4.13}, we have $\mathcal{F}_C(R)\bot[\mathcal{I}_C(R^{op})]^+$.
\end{enumerate}
Moreover, $\mathcal{A}_C(R^{op})$ is closed under coproducts and products by \cite[Proposition 4.2(a)]{HW07}.
It follows from Lemma \ref{4.14} and Theorem \ref{3.5} that
$$(\mathcal{A}_C(R^{op}),l\mathcal{G}(\mathcal{F}_C(R)))$$ is a coproduct-closed and product-closed duality pair
and $\mathcal{A}_C(R^{op})$ is covering and preenveloping in $\Mod R^{op}$.
\end{proof}

The following corollary was proved in \cite[Theorem 3.11]{EH} when $R$ is a commutative noetherian ring and $_RC_S={_RC_R}$.

\begin{corollary}\label{4.16}
If $S$ is a right coherent ring,
then $$(\mathcal{A}_C(R^{op}),\mathcal{A}_C(R^{op})^{\bot})$$ is a hereditary perfect cotorsion pair
and $\mathcal{A}_C(R^{op})$ is covering and preenveloping in $\Mod R^{op}$.
\end{corollary}

\begin{proof}
By \cite[Theorem 6.2]{HW07}, $\mathcal{A}_C(R^{op})$ is projectively resolving. Now the assertion follows from
Theorem \ref{4.15} and Lemma \ref{3.4}(3).
\end{proof}

As a consequence of Theorem \ref{4.15}, we also have the following

\begin{corollary}\label{4.17}
If $S$ is a right coherent and left perfect ring,
then $$(\mathcal{A}_C(R^{op}),\mathcal{B}_C(R))$$ is a coproduct-closed and product-closed duality pair.
\end{corollary}

\begin{proof}
Let $S$ is a right coherent and left perfect ring. By Lemma \ref{4.14}, we have
$$\mathcal{B}_C(R)=l\mathcal{G}(\mathcal{P}_C(R))=l\mathcal{G}(\mathcal{F}_C(R)).$$
Now the assertion follows from Theorem \ref{4.15}.
\end{proof}

\vspace{0.5cm}

{\bf Acknowledgements.}
This research was partially supported by NSFC (Grant No. 11571164) and the NSF of Shandong Province (Grant No. ZR2019QA015).
The authors thank the referees for the useful suggestions.

\end{document}